\documentclass[a4paper,reqno,12pt]{amsart} 
\usepackage{amsmath} 

\usepackage{amssymb}      

\usepackage{yhmath}
\usepackage{mathdots}
\usepackage{MnSymbol}

\usepackage{amsthm}      

\usepackage{tikz,amsmath}
\usetikzlibrary{arrows}

\usepackage{epsfig}      

\usepackage{graphicx}
\usepackage[normalem]{ulem}
\usepackage[all,line,
]{xy} 
\CompileMatrices 

   \newcommand{\Aff}{{\operatorname{Aff}}}
 
   \newcommand{\Hom}{\operatorname{Hom}}

\newcommand{\Ad}{\operatorname{Ad}}

\newcommand{\id}{\operatorname{id}} 
\newcommand{\Aut}{\operatorname{Aut}}

\newcommand{\Tr}{\operatorname{Tr}}

 \newcommand{\supp}{\operatorname{supp}}

\newcommand{\ev}{\operatorname{ev}}

   \theoremstyle{plain}
   \newtheorem{thm}{Theorem}[section]
   \newtheorem{prop}[thm]{Proposition}
   \newtheorem{lemma}[thm]{Lemma}  
   \newtheorem{cor}[thm]{Corollary}
   \theoremstyle{definition}
   
   \newtheorem{defn}[thm]{Definition}
   
   \theoremstyle{remark}

\newtheorem{TOM}[thm]{Property}

\usepackage{mdframed,xcolor}

\definecolor{mybgcolor}{gray}{0.8}
\definecolor{myframecolor}{rgb}{.647,.129,.149}

\mdfdefinestyle{mystyle}{
  usetwoside=false,
  skipabove=0.6em plus 0.8em minus 0.2em,
  skipbelow=0.6em plus 0.8em minus 0.2em,
  innerleftmargin=.25em,
  innerrightmargin=0.25em,
  innertopmargin=0.25em,
  innerbottommargin=0.25em,
  leftmargin=-.75em,
  rightmargin=-0em,
  topline=false,
  rightline=false,
  bottomline=false,
  leftline=false,
  backgroundcolor=mybgcolor,
  splittopskip=0.75em,
  splitbottomskip=0.25em,
  innerleftmargin=0.5em,
  leftline=true,
  linecolor=myframecolor,
  linewidth=0.25em,
}

\newmdenv[style=mystyle]{important}

\usepackage{lipsum}

   \numberwithin{equation}{section}

        \date{\today}

\title[KMS bundles]{On the Bundle of KMS state spaces for flows on a $\mathcal Z$-absorbing C*-algebra}
\author{ George A. Elliott, Yasuhiko Sato, and Klaus Thomsen}

\DeclareMathOperator\coker{coker}

\date{\today}

\address{Department of Mathematics, University of Toronto, Toronto, Ontario,
Canada \ M5S 2E4}
\email{elliott@math.toronto.edu}

\bigskip

\address{Graduate School of Mathematics, Kyushu University, 744 Motoka, Nishi-ku, Fukuoka, Japan}
\email{ysato@math.kyushu-u.ac.jp }

\bigskip

\address{Department of Mathematics, Aarhus University, Ny Munkegade, 8000 Aarhus C, Denmark}

\email{matkt@math.au.dk}

\begin{document}

\maketitle

\section{Introduction} The recent classification results for simple $C^*$-algebras have provided new tools for the construction of flows on $C^*$-algebras with a hitherto unseen richness and complexity in the configuration of KMS state spaces. The basic idea behind the construction of flows with specified KMS behaviour goes back to work by Bratteli, Elliott, Herman, and Kishimoto in \cite{BEH}, \cite{BEK1}, and \cite{BEK2} where such flows were constructed on various simple unital $C^*$-algebras. By using the classification results it is now possible to construct such flows on many given $C^*$-algebras. This was done for simple AF algebras in \cite{Th3} and \cite{ET} and it is the purpose here to do this with arbitrary simple algebras that absorb the Jiang-Su algebra tensorially, including in particular the purely infinite $C^*$-algebras that are classified by the Kirchberg-Phillips classification result. As in \cite{Th3} and \cite{ET} some of the key ingredients come from arguments developed by the second author and Matui.

Although little was known about flows that are not approximately inner on an AF-algebra, Kishimoto constructed a non approximately inner flow on a certain AF-algebra based on H. Lin's TAF classification theorem \cite{Ki3}. 
In \cite{Sa}, \cite{MS2}, and \cite{MS3}, a generalization of the classifiability permitted a construction of a non approximately inner flow on any UHF-algebra, which is known as a counter-example to the Powers-Sakai conjecture. The main result of this paper widens these previous constructions in the directions of both classifiable $C^*$-algebras and the variations of KMS spectra. As a consequence, we can construct uncountably many flows which are not approximately inner on any classifiable monotracial $C^*$-algebra up to weak cocycle conjugacy (Corollary \ref{21-10-21b}).

\section{Proper simplex bundles and flows on unital $C^*$-algebras}\label{2000}

We start with an abstract characterization of bundles of KMS state spaces introduced in \cite{ET}. Fix a second countable locally compact Hausdorff space $S$ and let $\pi$  be a continuous map from $S$ to $\mathbb R$. We shall say that the couple $(S, \pi)$ is a \emph{simplex bundle} if the inverse image $\pi^{-1}(t)$ is a compact metrizable Choquet simplex for any $t\in \mathbb R$ in the relative topology from $S$. Here we remark that $\pi$ is not necessarily surjective and the empty set is regarded as a simplex. When $(S, \pi)$ is a simplex bundle we shall denote by $\mathcal A(S, \pi)$ the set of continuous functions $f$ from $S$ to $\mathbb R$ such that 
the restriction $f|_{\pi^{-1}(t)}$ of $f$ to $\pi^{-1}(t)$ is affine for any $t\in\mathbb R$.
\begin{defn}\label{25-08-21} (See \cite{ET} and \cite{BEK2}.) A simplex bundle $(S, \pi)$ is \emph{proper}, if
\begin{itemize}
\item[(1)] $\pi$ is proper, i.e., $\pi^{-1}(K)$ is compact in $S$ for any compact subset $K$ of $\mathbb R$, 
\item[(2)] $\mathcal A(S,\pi)$ separates points of $S$, i.e., for any $x\neq y$ in $S$ there exists  $f \in\mathcal A(S,\pi)$ such that $f(x) \neq f(y)$.
\end{itemize}
\end{defn}
Note that when $S$ is compact, condition (1) is automatically satisfied and that $(S,\pi)$ is then a compact simplex bundle over $\mathbb R$ in the sense of \cite{BEK2}. Since we only consider bundles over $\mathbb R$ in this paper we shall call $(S,\pi)$ a \emph{compact simplex bundle} when $S$ is compact.  

Two proper simplex bundles $(S,\pi)$ and $(S',\pi')$ are \emph{isomorphic} when there is a homeomorphism $\phi : S \to S'$ such that $\pi' \circ \phi = \pi$ and $\phi: \pi^{-1}(\beta) \to {\pi'}^{-1}(\beta)$ is affine for all $\beta \in \mathbb R$.

The reason for introducing the concept of proper simplex bundle is that the collection of KMS state spaces for a flow on a unital separable $C^*$-algebra is a proper simplex bundle in a canonical way. To explain this we emphasize that all $C^*$-algebras in this paper are assumed to be separable and all traces and weights on a $C^*$-algebra are required to be non-zero, densely defined and lower semicontinuous. A flow $\theta = ( \theta_t)_{t \in \mathbb R}$ on a $C^*$-algebra $A$ is a continuous representation of $\mathbb R$ by automorphisms of $A$. Let $A$ be a $C^*$-algebra, $\theta$ a flow on $A$ and $\beta$ a real number. A $\beta$-KMS weight for $\theta$ is a weight $\omega$ on $A$ such that $\omega \circ \theta_t = \omega$ for all $t$, and 
\begin{equation}\label{27-10-20c}
\omega(a^*a) \ = \ \omega\left(\theta_{-\frac{i\beta}{2}}(a) \theta_{-\frac{i\beta}{2}}(a)^*\right), \ \ a \in D(\theta_{-\frac{i\beta}{2}})  .
\end{equation}
In particular, a $0$-KMS weight for $\theta$ is a $\theta$-invariant trace. 
 A bounded $\beta$-KMS weight is called a $\beta$-KMS functional and a $\beta$-KMS state when it is of norm one. For states alternative formulations of the KMS condition can be found in \cite{BR}.

 Assume that $A$ is unital. For each $\beta \in \mathbb R$ denote by $S^\theta_\beta$ the (possibly empty) set of $\beta$-KMS states for $\theta$, and by $E(A)$ the state space of $A$, a compact convex set in the weak* topology.
Set
$$
S^\theta = \left\{(\omega, \beta) \in E(A) \times \mathbb R: \ \omega \in S^\theta_\beta \right\} ,
$$
and equip $S^\theta$ with the relative topology inherited from the product topology of $E(A) \times \mathbb R$. Let $\pi^\theta : S^\theta \to \mathbb R$ denote the projection onto the second coordinate. It follows from general facts about  KMS states that $(S^\theta,\pi^\theta)$ is a proper simplex bundle, \emph{the KMS bundle of $\theta$}. See \cite{ET}.

\section{Compact simplex bundles as KMS bundles for flows on the  Jiang-Su algebra}

While we would like to construct a flow on the Jiang-Su algebra such that its bundle of KMS state spaces is an arbitrary proper simplex bundle for which the fiber $\pi^{-1}(0)$ over $0$  contains exactly one point, at present we only know how to do this when the bundle is compact. We fix therefore a compact simplex bundle $(S,\pi)$ such that $\pi^{-1}(0)$ contains exactly one point $\overline{o}$.

\subsection{Construction of a simple ordered abelian group}\label{3.1}

In what follows we denote the closed support of a real-valued function $f$ by $\supp f$. Let $\mathcal A_{0}(S,\pi)$ denote the set of elements $f \in \mathcal A(S,\pi)$ for which $ \overline{o} \notin \supp f$. Since the topology of $S$ is second countable we can choose a countable subgroup $G_{0}$ of $\mathcal A_{0}(S,\pi)$ with the following density property:

\begin{TOM}\label{07-09-21} For all $\epsilon > 0$ and all $f \in \mathcal A(S,\pi)$ such that $f(\overline{o}) = 0$, there is an element $g \in G_{0}$ such that $\sup_{x \in S} |f(x) -g(x)| < \epsilon$. 
\end{TOM}

Enlarging $G_0$, we may suppose that the functions
$$
{e^{n \pi}}{(1-e^{-\pi})^m}f, \ \ n,m \in \mathbb Z,
$$
all are in $G_0$ when $f$ is. Set
$$
G = \left( \bigoplus_{\mathbb Z} \mathbb Z\right) \oplus G_0,
$$ 
and define ${L} : G  \to \mathcal A(S,\pi)$ by
$$
{L}\left(\xi,  g\right)(x) = g(x) + \sum_{n \in \mathbb Z} z_ne^{n\pi(x)} ,
$$ 
where $\xi = (z_n)_{n \in \mathbb Z}  \in \bigoplus_\mathbb Z \mathbb Z $. 
Set
$$
G^+ = \left\{ (\xi, g) \in G : \ L(\xi ,g)(x) > 0, \ x \in S\right\} \cup \{0\} .
$$
The proof of the following lemma is straightforward. The last statement uses that $S$ is compact.
\begin{lemma}\label{09-10-21} $G = G^+ - G^+, \ G^+ \cap (-G^+) = \{0\}$, and every non-zero element of $G^+$ is an order unit for $(G,G^+)$.
\end{lemma}

In short, this lemma says that $(G,G^+)$ is a simple ordered abelian group. Note that 
$$
\mathbb Q \otimes_\mathbb Z G =  \left(\bigoplus_\mathbb Z \mathbb Q \right) \oplus \mathbb QG_0
$$
and that $L$ extends to a $\mathbb Q$-linear map $\overline{L} : \mathbb Q \otimes_\mathbb Z G \to \mathcal A(S,\pi)$. We set
$$
(\mathbb Q \otimes_\mathbb Z G)^+ = \left\{ (\xi, g) \in \mathbb Q \otimes_\mathbb Z G : \ \overline{L}(\xi ,g)(x) > 0, \  x \in S\right\} \cup \{0\} .
$$

\begin{lemma}\label{09-10-21a} $( \mathbb Q \otimes_\mathbb Z G ,(\mathbb Q \otimes_\mathbb Z G )^+)$ is a simple ordered abelian group with the strong Riesz interpolation property: If $g_i, k_j\in  \mathbb Q \otimes_\mathbb Z G$ and $g_i < k_j$ for $i,j \in \{1,2\}$, then there is an element $h \in \mathbb Q \otimes_\mathbb Z G$ such that $g_i < h < k_j$ for all $i,j \in \{1,2\}$.
\end{lemma}
\begin{proof} Only the interpolation property is not straightforward, so assume that $g_i, k_j\in  \mathbb Q \otimes_\mathbb Z G$ and $g_i < k_j$ for $i,j \in \{1,2\}$. Choose $q \in \mathbb Q$ such that $g_i(\overline{o}) < q < k_j(\overline{o})$ for all $i,j$. It follows from Lemma 2.2 of \cite{BEK2} that there is an element $h_0 \in \mathcal A(S,\pi)$ such that $h_0(\overline{o}) = q$ and $g_i(x) < h_0(x) < k_j(x)$ for all $i,j$ and all $x \in S$. Let $\delta > 0$ be smaller than $k_j(x)-h_0(x)$ and $h_0(x)-g_i(x)$ for all $i,j$ and all $x \in S$. Thanks to Property \ref{07-09-21} there is $g \in G_0$ such that $\left| h_0(x) - q - g(x)\right| < \delta$ for all $x \in S$. 
For $t \in \mathbb Q$ denote by $t^{(0)}$ the element of $\bigoplus_\mathbb Z \mathbb Z$ given by
$$
t^{(0)}_n = \begin{cases} t, & \ n = 0, \\ 0, & \ n \neq 0. \end{cases}
$$
Then $ h = (q^{(0)},g) \in G$ has the desired properties.
\end{proof}

In the notation of the last proof, set $v = (1^{(0)},0) \in G^+$. Then $v$ is an order unit in both $(G,G^+)$ and $( \mathbb Q \otimes_\mathbb Z G ,(\mathbb Q \otimes_\mathbb Z G )^+)$. Denote by $S(G)$ and $S(\mathbb Q \otimes_\mathbb Z G)$ the corresponding state spaces. The restriction map 
\begin{equation}\label{09-10-21c} 
S(\mathbb Q \otimes_\mathbb Z G) \to S(G)
\end{equation}
is clearly an affine homeomorphism and we conclude therefore from Lemma \ref{09-10-21a} that $S(G)$ is a Choquet simplex by Proposition 1.7 of \cite{EHS}.

Let $\mathcal Z$ denote the Jiang-Su algebra, \cite{JS}. In what follows a $C^*$-algebra $A$ will be called $\mathcal Z$-stable when $A \otimes \mathcal Z \simeq A$. Let $F_0,F_1$ be finite-dimensional $C^*$-algebras and $\varphi_0, \varphi_1 : F_0 \to F_1$ unital $*$-homomorphisms. The $C^*$-algebra 
$$
\left\{ (a,f) \in F_0 \oplus (C[0,1] \otimes F_1): \ \varphi_i(a) = f(i), \ i =0,1 \right\}
$$
will be called a \emph{building block}. In what follows we shall denote by $T(A)$ the tracial state space of a unital $C^*$-algebra $A$.

\begin{prop}\label{30-09-21bxx} (\cite{E3},\cite{Th1},\cite{GLN1}.) There is a simple unital $\mathcal Z$-stable $C^*$-algebra $E$ which is an inductive limit of building blocks with the following properties. 
\begin{itemize}
\item[(1)] $K_1(E) = 0$.
\item[(2)] The canonical map $\tau : T(E) \to S(K_0(E),[1])$ is bijective.
\item[(3)] There is an isomorphism $\phi : K_0(E) \to G$ of ordered abelian groups such that $\phi([1]) = v$.
\end{itemize}
\end{prop}
\begin{proof} This follows from Corollary 13.51 of \cite{GLN1} since the quadruple
$$
((G,G^+,v),0, S(G), \id)$$ 
is an Elliott invariant with the properties required in that statement.
\end{proof}

In the sequel we fix $E$ as in Proposition \ref{30-09-21bxx} and just write
$$
(K_0(E),K_0(E)^+,[1],T(E)) = (G,G^+,v,S(G)).
$$
We shall denote by $\mathbb K$ the $C^*$-algebra of compact operators on an infinite-dimensional separable Hilbert space and fix a non-zero minimal projection $e_{11} \in \mathbb K$. Denote by $\Hom_+(G,\mathbb R)$ the cone of non-zero positive homomorphisms from $G$ to $\mathbb R$ and by $\mathcal T(E \otimes \mathbb K)$ the cone of densely defined lower semicontinuous traces on $E\otimes \mathbb K$. Recall that, as $E$ is unital and simple, the restriction map $\tau \mapsto \tau|_{E \otimes e_{11}}$ is bijective.

Define the automorphism $\alpha$ of $(G,G^+)$ by
$$
\alpha (\xi,f) = ( \sigma(\xi), e^{-\pi}f) , 
$$
where $\sigma \in \Aut \left(\bigoplus_\mathbb Z \mathbb Z\right)$ is the shift;
$$
\sigma(\xi)_n = \xi_{n+1}.
$$

\begin{lemma}\label{14-10-21} There is an automorphism $\gamma \in \Aut E \otimes \mathbb K$ such that $\gamma_* = \alpha$.
\end{lemma}

\begin{proof} Set $p_1 = 1 \otimes e_{11}$. There is a projection $p_2 \in E\otimes \mathbb K$ such that $\alpha([p_1]) = [p_2]$ in $K_0(E)$. Consider the homomorphism $s : E \otimes \mathbb K \to E\otimes \mathbb K \otimes \mathbb K$, $a \mapsto  a\otimes e_{11}$. It is well known that $s$ is homotopic in $\Hom( E \otimes \mathbb K, E \otimes \mathbb K \otimes \mathbb K)$ to a $*$-isomorphism $\mu :E \otimes \mathbb K \to E \otimes \mathbb K \otimes \mathbb K$. In particular, ${s}_* : K_0(E) \to K_0(E \otimes \mathbb K)$ is an isomorphism. It follows from \cite{B} that there are isometries $V_i \in M(E \otimes \mathbb K \otimes \mathbb K)$ such that $V_iV_i^* = p_i \otimes 1_{\mathbb K}$. Then $\Ad V_i \circ s : E \otimes \mathbb K \to p_i(E\otimes \mathbb K)p_i \otimes \mathbb K$ is $*$-homomorphism homotopic to a $*$-isomorphism and we get therefore an isomorphism 
$$
 \alpha'  : K_0(p_1(E \otimes \mathbb K)p_1) \to K_0(p_2(E\otimes \mathbb K)p_2)
$$
of ordered groups when we set 
$$
\alpha' = (\Ad V_2 \circ s)_*\circ \alpha \circ (\Ad V_1 \circ s)_*^{-1}.
$$
Note that $V_i(p_i \otimes e_{11})$ is a partial isometry in $p_i(E\otimes \mathbb K)p_i \otimes \mathbb K$ giving a Murray-von Neumann equivalence $V_i(p_i \otimes e_{11})V_i^* \sim p_i \otimes e_{11}$. Thus,
$$
(\Ad V_i \circ s)_*^{-1}[p_i \otimes e_{11}] = [p_i]
$$
and hence, on setting $i =1$, $\alpha'([p_1 \otimes e_{11}]) = [p_2 \otimes e_{11}]$. Thus $\alpha'$ is an isomorphism 
\begin{align*}
&(K_0(p_1(E \otimes \mathbb K)p_1), K_0(p_1(E\otimes \mathbb K)p_1)^+,[1_{p_1(E \otimes \mathbb K)p_1}]) \\
& \ \ \ \ \ \ \ \ \ \ \ \ \ \ \ \to  (K_0(p_2(E \otimes \mathbb K)p_2), K_0(p_2(E\otimes \mathbb K)p_2)^+,[1_{p_2(E \otimes \mathbb K)p_2}])
\end{align*}
of ordered groups with order unit. 

Note that the condition (2) of Proposition \ref{30-09-21bxx} is equivalent to the statement that the canonical map $\mathcal T(E \otimes \mathbb K)$ to positive functionals on $K_0(E)=K_0(E \otimes \mathbb K)$ is bijective. In other words, traces on $E \otimes \mathbb K$ are determined by their values on projections. Since $p_i(E \otimes \mathbb K)p_i \otimes \mathbb K \simeq E \otimes \mathbb K$, the same is true with $p_i(E \otimes \mathbb K)p_i$ in place of $E, \ i = 1,2$. It follows immediately that there is a unique affine homeomorphism $\chi : T(p_1(E \otimes \mathbb K)p_1) \to  T(p_2(E \otimes \mathbb K)p_2)$ compatible with $\alpha'$ in the natural sense, i.e.,
$$
\chi(\omega)_*(\alpha'(x)) = \omega_*(x)
$$
for all $\omega \in T(p_1(E \otimes \mathbb K)p_1)$ and all $x \in K_0(p_1(E\otimes \mathbb K)p_1)$.

 It follows from \cite{B} that $p_i(E\otimes \mathbb K)p_i$ is stably isomorphic to $E$ and hence from (1) in Proposition \ref{30-09-21bxx} that $K_1(p_i(E\otimes \mathbb K)p_i) = 0$. So we see that the pair $(\alpha',\chi)$ is an isomorphism of the Elliott invariant of $p_1(E \otimes \mathbb K)p_1$ onto that of $p_2(E \otimes \mathbb K)p_2$.
 
  We wish to apply Theorem 2.7 of \cite{EGLN} to conclude that the isomorphism $(\alpha',\chi)$ arises from a $*$-isomorphism between the algebras. This means showing that  $p_1(E \otimes \mathbb K)p_1$ and $p_2(E \otimes \mathbb K)p_2$ belong to the class $\mathcal N_1$ of \cite{EGLN}, i.e., have rational generalized tracial rank at most one, and in addition are amenable, satisfy the UCT, and are $\mathcal Z$-stable.
  
  For convenience, set $p_i(E \otimes \mathbb K)p_i =E_i, \ i =1,2$. Note that, by Lemma 3.19 of \cite{GLN1}, both $E_1$ and $E_2$ are simple inductive limits of building blocks since $E$ is. Both $E_1$ and $E_2$ are hereditary sub-$C^*$-algebras of $E\otimes \mathbb K$ and hence $\mathcal Z$-stable by Corollary 3.1 of \cite{TW}, as $E$ is $\mathcal Z$-stable. Hence by Theorem A of \cite{ENST}, $E_1$ and $E_2$ have finite decomposition rank (in fact, $dr(E_i) \leq 2, \ i = 1,2$). By 22.3.5 (e) and 23.1.1 of \cite{Bl}, the Universal Coefficient Theorem (UCT) holds for inductive limits of type I algebras and so $E_1$ and $E_2$ satisfy the UCT. Hence by Theorem 1.1 of \cite{EGLN}, $E_1$ and $E_2$ belong to $\mathcal N_1$ -- and so the second part of Theorem 2.7 of \cite{EGLN} applies, and the isomorphism between the Elliott invariants is determined by a $*$-isomorphism $\rho : E_1 \to E_2$.

Having $\rho$, we set $
\gamma = \mu^{-1} \circ \Ad V_2^* \circ (\rho \otimes \id_{\mathbb K}) \circ \Ad V_1 \circ \mu \in \Aut E \otimes \mathbb K$.
Then 
\begin{align*}
&\gamma_* = (\mu^{-1})_* \circ (\Ad V_2^*)_* \circ \rho_* \circ  (\Ad V_1)_* \circ \mu_*\\
& =(s_*)^{-1} \circ (\Ad V_2^*)_* \circ \alpha' \circ (\Ad V_1 \circ s)_*\\
& = (\Ad V_2 \circ s)_*^{-1} \circ \alpha' \circ (\Ad V_1 \circ s)_* = \alpha .\\
\end{align*}
\end{proof}

The following lemma follows from \cite{Sa}, \cite{MS2} and \cite{MS3}; see also the second step in the proof of Lemma 3.4 of \cite{Th2}.

\begin{lemma}\label{30-09-21dx} The automorphism $\gamma$ can be chosen to have the following additional properties.
\begin{itemize}
\item[(A)] The restriction map 
$\mu \ \mapsto \ \mu|_{E\otimes \mathbb K}$
is a bijection from traces $\mu $ in $\mathcal T(E \otimes \mathbb K) \rtimes_{\gamma} \mathbb Z)$ onto the $\gamma$-invariant traces in $\mathcal T(E\otimes \mathbb K)$, and
\item[(B)] $(E\otimes \mathbb K) \rtimes_{\gamma} \mathbb Z$ is $\mathcal Z$-stable; that is $((E \otimes \mathbb K) \rtimes_{\gamma} \mathbb Z)\otimes \mathcal Z \simeq (E\otimes \mathbb K) \rtimes_{\gamma} \mathbb Z$ where $\mathcal Z$ denotes the Jiang-Su algebra, \cite{JS}.
\end{itemize} 
\end{lemma}

We assume in what follows that $\gamma \in \Aut E \otimes \mathbb K$ is chosen such that $\gamma_* = \alpha$ and such that (A) and (B) of Lemma \ref{30-09-21dx} hold. Set
$$
C = (E \otimes \mathbb K)\rtimes_\gamma \mathbb Z  \ .
$$
Since $\gamma_*^k = \alpha^k \neq \id$ when $k \neq 0$, no non-trivial power of $\gamma$ is inner. Since $E \otimes \mathbb K$ is simple it follows from Theorem 3.1 of \cite{Ki1} that $C$ is simple. 
It follows from the Pimsner-Voiculescu exact sequence, \cite{PV}, that we can identify $K_0(C)$, as a group, with the quotient
$$
G/(\id - \alpha)(G)  ,
$$
in such a way that the map $\iota_* : K_0(E) \to K_0(C)$ induced by the inclusion $\iota : E \otimes \mathbb K \to C$ becomes the quotient map
$$
q : G \to  G/(\id - \alpha)(G) .
$$
Define $\Sigma_0 : G \to \mathbb Z$ by
$$
\Sigma_0((z_n)_{n \in \mathbb Z}, f) = \sum_{n \in \mathbb Z} z_n. 
$$

\begin{lemma}\label{06-10-21c} $\ker \Sigma_0 = (\id - \alpha)(G)$.
\end{lemma}
\begin{proof}  If $((z_n)_{n \in \mathbb Z}, f) \in \ker \Sigma_0$ it follows as in the proof of Lemma 4.6 of \cite{Th3} that there is an element $\xi \in \bigoplus_\mathbb Z \mathbb Z$ such that $(z_n)_{n \in \mathbb Z} = (\id - \sigma)(\xi)$. It follows immediately from the conditions imposed on $G_0$ that $f = (1 -e^{-\pi})g$ for some $g \in G_0$. Then $(\id -\alpha)(\xi,g) = ((z_n)_{n \in \mathbb Z}, f)$ showing that $\ker \Sigma_0 \subseteq (\id - \alpha)(G)$. The reverse inclusion is trivial.
\end{proof}

It follows from Lemma \ref{06-10-21c} that $\Sigma_0$ induces an isomorphism
$$
\Sigma : K_0(C) = G/(\id - \alpha)(G)  \to \mathbb Z \ 
$$
such that $\Sigma \circ q = \Sigma_0$. Every point $s \in S$ defines by evaluation a state $\ev_s$ on $\mathcal A(S,\pi)$. Note that
$$
\ev_s \circ L\circ \alpha = e^{- \beta} \ev_s \circ L
$$
where $\beta = \pi(s)$, and $L$ is as defined in Section \ref{3.1}.

\begin{lemma}\label{27-08-21xxx}
Let $\phi \in S(G)$ be a state of $(G,G^+,v)$ with the property that $\phi \circ \alpha = s \phi$ for some $s > 0$. Set $\beta = -\log s$. There is a unique point $s \in \pi^{-1}(\beta)$ such that $\phi = \ev_s\circ L$.
\end{lemma} 
\begin{proof} Extend $\phi$ to a $\mathbb Q$-linear state $\overline{\phi}$ of $(\mathbb Q \otimes_\mathbb Z G, (\mathbb Q \otimes_\mathbb Z G)^+,v)$. If $(\xi,g) \in \mathbb Q \otimes_\mathbb Z G$ and $\overline{L}(\xi,g) = 0$ we have that $\frac{1}{n}v \pm (\xi,g) \in (\mathbb Q \otimes_\mathbb Z G)^+$ and hence 
$$
- \frac{1}{n} \leq \overline{\phi}(\xi,g) \leq \frac{1}{n}
$$
for all $n \in \mathbb N$, i.e. $\overline{\phi}(\xi,g) =0$. It follows that there is a $\mathbb Q$-linear map $\psi : \overline{L}(\mathbb Q \otimes_\mathbb Z G) \to \mathbb R$ such that $\psi \circ \overline{L} = \overline{\phi}$. If $f \in \overline{L}(\mathbb Q \otimes_\mathbb Z G)$, $n,m \in \mathbb N$ and $\left|f(x)\right| < \frac{n}{m}$ for all $x \in S$, there is an element $w \in \mathbb Q \otimes_\mathbb Z G$ such that $\overline{L}(w) = f$ and $-n v < m w < nv$ in $\mathbb Q \otimes_\mathbb Z G$, implying that 
$$
\left|\psi(f)\right| \leq \frac{n}{m}.
$$
It follows that $\left|\psi(f)\right| \leq \left\|f\right\|$. It is an easy consequence of Property \ref{07-09-21} that $\overline{L}(\mathbb Q \otimes_\mathbb Z G)$ is dense in $\mathcal A(S,\pi)$ and it follows therefore that $\psi$ extends to a linear norm-contractive map $\psi : \mathcal A(S,\pi) \to \mathbb R$. Using the Hahn-Banach theorem we can extend $\psi$ further to a norm-contractive linear map on the space $C_\mathbb R(S)$ of continuous real-valued functions on $S$. Since $\psi(1) = \phi(v) =1$ this extension is positive and it follows that there is a Borel probability measure $m$ on $S$ such that $\psi$ is given by integration with respect to $m$. Then
\begin{align*}
&\int_S e^{-\pi(x)} L(\xi,g)(x) \ \mathrm{d} m(x) =  \int_S  L(\sigma(\xi),e^{-\pi}g)(x) \ \mathrm{d} m(x) \\
&= \phi\circ \alpha(\xi,g) = s \phi(\xi,g) = s\int_S  L(\xi,g)(x) \ \mathrm{d} m(x)
\end{align*}
for all $(\xi,g) \in G$. Using that $\overline{L}(\mathbb Q \otimes_\mathbb Z G) = \mathbb Q {L}(G) $ is dense in $\mathcal A(S,\pi)$ it follows that
$$
\int_S e^{-\pi(x)} f(x) \ \mathrm{d} m(x) =  s\int_S  f(x) \ \mathrm{d} m(x)
$$
for all $f \in \mathcal A(S,\pi)$. Let $g \in C_\mathbb R(\pi(S))$. Then $g \circ \pi \in \mathcal A(S,\pi)$ and hence
\begin{align*}
&\int_{\pi(S)} e^{-t} g(t) \ \mathrm{d} m \circ \pi^{-1}(t) = \int_S e^{-\pi(x)} g \circ \pi(x) \ \mathrm{d} m(x) \\
&= s  \int_S  g \circ \pi(x) \ \mathrm{d} m(x) = s\int_{\pi(S)} g(t) \ \mathrm{d} m \circ \pi^{-1}(t).
\end{align*}
Since this holds for all $g \in  C_\mathbb R(\pi(S))$ it follows that $-\log s \in \pi(S)$ and that $m\circ \pi^{-1}$ is concentrated at $\beta = -\log s$, i.e., $m$ is concentrated on $\pi^{-1}(\beta)$. In follows that $\psi$ factorises through $\Aff \pi^{-1}(\beta)$, i.e., $\psi = \tilde{\psi} \circ r$ where $\tilde{\psi} : \Aff \pi^{-1}(\beta) \to \mathbb R$ is linear and $r : \mathcal A(S,\pi) \to \Aff \pi^{-1}(\beta)$ is given by restriction. When $a \in \Aff \pi^{-1}(\beta)$ there is $\tilde{a} \in \mathcal A(S,\pi)$ such that $\tilde{a}|_{\pi^{-1}(\beta)} = a$ by Lemma 2.2 of \cite{BEK2}. If, in addition, $a \geq 0$ and $\epsilon > 0$ it follows from Lemma 2.3 of \cite{BEK2} that we can choose $\tilde{a}$ such that $-\epsilon < \tilde{a}$ in $\mathcal A(S,\pi)$. It follows in this way that $\tilde{\psi}$ is a state on $\Aff \pi^{-1}(\beta)$. Since every state of $\Aff \pi^{-1}(\beta)$ is given by evaluation at a point in $\pi^{-1}(\beta)$ we obtain a point $s \in \pi^{-1}(\beta)$ such that $\tilde{\psi} = \ev_s$. It follows that $\psi = \ev_s$ and hence $\phi = \ev_s \circ L$. The point $s$ is unique because $\mathcal A(S,\pi)$ separates the points of $S$ and $L(G)$ linearly spans a dense set in  $\mathcal A(S,\pi)$. 
\end{proof}

Reformulating Lemma \ref{27-08-21xxx} in terms of $E$, we obtain the following statement.

\begin{lemma}\label{09-10-21i} There is a unique $\gamma$-invariant trace $\tau^0$ on $E \otimes \mathbb K$ with the property that $\tau^0(1\otimes e_{11}) = 1$. It is determined by the condition that $\tau^0_* = \ev_{\overline{o}}\circ L$.
\end{lemma}

\begin{lemma}\label{02-09-21} $\Sigma(K_0(C)^+) = \mathbb N \cup \{0\}$.
\end{lemma}
\begin{proof} In the sequel we shall denote by $P$ the canonical conditional expectation $P : C \to E \otimes \mathbb K$. Let $n \in \mathbb N$. Then $(n^{(0)},0) \in G^+$, $q((n^{(0)},0)) \in K_0(C)^+$ and $\Sigma(q( (n^{(0)},0))) = n$. It follows that $\mathbb N \subseteq \Sigma(K_0(C)^+)$. Let $x \in K_0(C)^+$ and write $x = q((z_n)_{n \in \mathbb Z},g)$ for some $((z_n)_{n \in \mathbb Z},g) \in G$. Since $x \in K_0(C)^+$,
$$
\left(\tau^0 \circ P\right)_*(x) \geq 0 ,
$$
where $\tau^0$ is the $\gamma$-invariant trace on $E \otimes \mathbb K$ from Lemma \ref{09-10-21i}.
Since 
\begin{align*}
&\left(\tau^0 \circ P\right)_*(x) = \tau^0_*((z_n)_{n \in \mathbb Z},g)) = \ev_{\overline{o}} \circ L((z_n)_{n \in \mathbb Z},g)) =\sum_{n\in \mathbb Z} z_n \in \mathbb Z,
\end{align*}
it follows that $\Sigma(x) \in \mathbb N$.
\end{proof}

\begin{lemma}\label{02-09-21cx} $K_1(C) = 0$.
\end{lemma}
\begin{proof} To establish this from the Pimsner-Voiculescu exact sequence, \cite{PV}, we must show that $\id -\alpha$ is injective, which is easy: If $\xi \in \bigoplus_\mathbb Z \mathbb Z$ and $\sigma(\xi) = \xi$ it follows immediately that $\xi =0$, and if $g \in G_0$ and $e^{-\pi}g = g$ it follows that $g =0$ because $g$ is supported away from $\overline{o}$.
\end{proof}

Let $p = 1 \otimes e_{11} \in E \otimes \mathbb K$.

\begin{lemma}\label{07-10-21} $pCp$ is $*$-isomorphic to the Jiang-Su algebra $\mathcal Z$.
\end{lemma}
\begin{proof} By combining (A) of Lemma \ref{30-09-21dx} with Proposition 4.7 of \cite{CP} and Lemma \ref{09-10-21i} above it follows that $pCp$ has exactly one trace state. By Lemma \ref{02-09-21} the isomorphism $\Sigma : K_0(C) \to \mathbb Z$ is an isomorphism of ordered groups. Note that $\Sigma([p]) = 1$. Since $(K_0(pCp), K_0(pCp)^+) = (K_0(C), K_0(C)^+)$ and $K_1(pCp) = K_1(C) = 0$ by Lemma \ref{02-09-21cx} we see that $pCp$ has the same Elliott invariant as $\mathcal Z$. By (B) of Lemma \ref{30-09-21dx} $C$ is $\mathcal Z$-stable and so by 3.1 of \cite{TW}, $pCp$ is $\mathcal Z$-stable. Thus, $pCp$ and $\mathcal Z$ are both simple, unital, nuclear, $\mathcal Z$-stable, and satisfy the UCT (see 23.1.1 and 22.3.5 of \cite{Bl}), and have a unique tracial state. By Corollary 6.2 of \cite{MS2} and Corollary 4.6 of \cite{Ro2} they are $*$-isomorphic. See Remark 4.12 of \cite{Th3}.
\end{proof}

\subsection{Flows on the Jiang-Su algebra}

We consider the dual action on $C = (E \otimes \mathbb K)  \rtimes_\gamma \mathbb Z$ as a $2\pi$-periodic flow and we denote by $\theta$ the restriction of this flow to $pCp$.

\begin{lemma}\label{02-09-21dx} The KMS bundle $(S^\theta,\pi^\theta)$ of $\theta$ is isomorphic to $(S,\pi)$.
\end{lemma}
\begin{proof} Let $(\omega,\beta) \in S^\theta$. By Remark 3.3 of \cite{LN} $\omega$ extends uniquely to a $\beta$-KMS weight $\widehat{\omega}$ on $C$, and by Lemma 3.1 of \cite{Th3} the restriction $\widehat{\omega}|_{E \otimes \mathbb K}$ is a trace on $E \otimes \mathbb K$ such that $\widehat{\omega}|_{E \otimes \mathbb K} \circ \gamma = e^{-\beta} \widehat{\omega}|_{E \otimes \mathbb K}$. Since $\widehat{\omega}(1 \otimes e_{11}) = 1$ it follows from Lemma \ref{27-08-21xxx} that $\left(\widehat{\omega}|_{E \otimes \mathbb K}\right)_* = \ev_s \circ L\in S(G)$ for some $s \in \pi^{-1}(\beta)$. This gives us a map $\xi : S^{\theta} \to S$ such that $\xi(\omega,\pi) =s$. Note that $\pi \circ \xi = \pi^{\theta}$. To see that the map is surjective, let $s \in S$ and set $\beta = \pi(s)$. Then $\ev_s \circ L \in S(G)$ and $\ev_s \circ L \circ \alpha= e^{-\beta} \ev_s \circ L$ so Lemma \ref{09-10-21i} gives us a trace $\rho$ on $E \otimes \mathbb K$ such that $\rho _*= \ev_s \circ L$ and $\rho \circ \gamma = e^{-\beta} \tau$. By Lemma 3.1 of \cite{Th3} the weight $\rho \circ P$ is a $\beta$-KMS weight for the dual action on $C$. Since $\rho \circ P(1\otimes e_{11}) = \rho_*(v) = \ev_s \circ L(v) = 1$, the restriction of $\rho\circ P$ to $pCp$ is $\beta$-KMS state for $\theta$ and it is clear that  $\xi(\rho \circ P|_{pCp},\beta) = s$. To see that the map is injective, let $(\omega^i,\beta^i), i = 1,2$, be elements of $S^\theta$ such that $\xi(\omega^1,\beta^1) = \xi(\omega^2,\xi^2) =s$. Since $s \in \pi^{-1}(\beta^1) \cap \pi^{-1}(\beta^2)$ it follows that $\beta^1 = \beta^2 = \beta$. Since $\left(\widehat{\omega^1}|_{E \otimes \mathbb K}\right)_* = \ev_s \circ L =  \left(\widehat{\omega^2}|_{E \otimes \mathbb K}\right)_*$ it follows by Proposition \ref{30-09-21bxx} that $\widehat{\omega^1}|_{E \otimes \mathbb K} = \widehat{\omega^2}|_{E \otimes \mathbb K}$. By Lemma 3.1 of \cite{Th3} this implies that $\widehat{\omega^1} = \widehat{\omega^2}$ and hence that $\omega^1 = \omega^2$. It follows that $\xi$ is a bijection. By Proposition \ref{30-09-21bxx} above and Lemma 3.1 of \cite{Th3} we have the following formula for the inverse $\xi^{-1}$:
$$
\xi^{-1}(s) = \left( \left( \tau^{-1}(ev_s \circ L) \otimes \Tr_{\mathbb K}\right) \circ P|_{pCp}, \ \pi(s)\right)  .
$$
It follows that $\xi^{-1}$ is continuous and therefore a homeomorphism.

\end{proof}

Combining Lemma \ref{07-10-21} and Lemma \ref{02-09-21dx} we obtain our main result:

\begin{thm}\label{30-09-21e} Let $(S,\pi)$ be a compact simplex bundle such that $\pi^{-1}(0)$ consists of exactly one point. There is a $2 \pi$-periodic flow $\theta$ on the Jiang-Su algebra whose KMS bundle is isomorphic to $(S,{\pi})$.
\end{thm}

From Theorem \ref{30-09-21e} we obtain the following corollary; see Corollary 3.2 of \cite{ET}.

\begin{cor}\label{10-10-21} Let $K$ be a compact subset of real numbers containing $0$.
\begin{itemize}
\item There is a $2 \pi$-periodic flow on the Jiang-Su algebra whose KMS spectrum is $K$ and such that there is a unique $\beta$-KMS state for all $\beta \in K$.
\item There is a $2\pi$-periodic flow $\theta$ on the Jiang-Su algebra whose KMS spectrum is $K$ and such that $S^\theta_\beta$ is not affinely homeomorphic to $S^\theta_{\beta'}$ when $\beta,\beta' \in K \backslash \{0\}$ and $\beta \neq \beta'$.
\end{itemize}
\end{cor}

 \section{Flows that are not approximately inner}
 
 In this section we shall use Theorem \ref{30-09-21e} to show that flows that are not approximately inner exist in abundance.

\begin{thm}\label{21-10-21a}
Let $A$ be a unital separable $C^*$-algebra with a unique trace state. Suppose that 
$A$ absorbs the Jiang-Su algebra tensorially. 
Then for any compact simplex bundle $(S, \pi)$ with $\pi^{-1}(0)$ consisting of exactly one point,
there exists a $2\pi$-periodic flow on $A$ whose KMS bundle is isomorphic to $(S, \pi)$.
\end{thm}

\begin{proof} Let $\mathcal Z$ be the Jiang-Su algebra. By Theorem \ref{30-09-21e} there is a flow $\theta$ on $\mathcal Z$ whose KMS bundle $(S^\theta,\pi^\theta)$ is isomorphic to $(S,\pi)$. Set
$$
\theta'_t = \id_A \otimes \theta_t
$$
to get a flow $\theta'$ on $A \otimes \mathcal Z$. Let $\tau$ be the trace state of $A$. Then
\begin{equation}\label{21-10-21}
\omega \mapsto \tau \otimes \omega
\end{equation}
defines a continuous map $S^\theta \to S^{\theta'}$ of KMS bundles which is clearly injective. To show that it is surjective, let $\omega'$ be a $\beta$-KMS state for $\theta'$. Let $a_i \in A, i = 1,2, \ b \in \mathcal Z$. Since $a_1 \otimes 1$ is fixed by $\theta'$ we find that
$$
\omega'(a_1a_2\otimes b) = \omega'((a_1 \otimes 1)(a_2 \otimes b)) = \omega'((a_2\otimes b)(a_1 \otimes 1)) = \omega'(a_2a_1 \otimes b) .
$$
Thus, $a \mapsto \omega'(a\otimes b)$ is a bounded trace functional on $A$ and hence $\omega'(a\otimes b) = \tau(a)\omega(b)$ for all $a \in A$ and some $\omega(b) \in \mathbb C$. It is straightforward to see that $b \mapsto \omega(b)$ is a state on $\mathcal Z$ and in fact a $\beta$-KMS state for $\theta$ since $\omega'$ is a $\beta$-KMS state for $\theta'$. This shows that the map \eqref{21-10-21} is surjective and hence a homeomorphism of KMS bundles. By moving the flow $\theta'$ onto $A$ via a $*$-isomorphism $A \simeq A \otimes \mathcal Z$, we obtain the desired flow on $A$.

\end{proof}

 For the statement of the following corollary we say that two flows $\alpha$ and $\mu$ on the same unital $C^*$-algebras $A$ are \emph{weakly cocycle-conjugate} when there are an automorphism $\gamma \in \Aut A$, a non-zero real number $\lambda \in \mathbb R \backslash \{0\}$ and continuous path $\{u_t\}_{t \in \mathbb R}$ of unitaries in $A$ such that
\begin{itemize}
\item[{}] $u_s\alpha_{\lambda s}(u_t) = u_{s+t}$ and
\item[{}] $\gamma \circ \mu_t \circ \gamma^{-1} = \Ad u_t \circ \alpha_{\lambda t}$
\end{itemize}
for all $s,t \in \mathbb R$. Weak cocycle-conjugacy is an equivalence relation on flows. Recall from \cite{PS} that a flow $\alpha$ is \emph{approximately inner} when there is a sequence $\{h_n\}$ of self-adjoint elements of $A$ such that 
$$
\lim_{n \to \infty} \left\|\alpha_t(a) - \Ad e^{ith_n}(a)\right\| = 0
$$
uniformly on compact subsets of $\mathbb R$ for all $a \in A$. It follows from Corollary 1.4 of \cite{Ki2} that a flow $\alpha$ is approximately inner if and only if all flows weakly cocycle-conjugate to $\alpha$ are approximately inner.

\begin{cor}\label{21-10-21b} Let $A$ be a unital separable $C^*$-algebra with a unique trace state. Suppose that 
$A$ absorbs the Jiang-Su algebra tensorially. There are uncountably many weak cocycle-conjugacy classes of flows on $A$ that are not approximately inner.
\end{cor}
\begin{proof} Note that weak cocycle-conjugacy, as defined in the preceding paragraph, is the combination of cocycle-conjugacy as defined for example in \cite{Ki2} and scaling, where $\alpha_t$ is replaced by $\alpha_{\lambda t}$. It follows from Proposition 2.1 of \cite{Ki2} and Section 4.1 of \cite{Th2} that cocycle-conjugate flows $\alpha$ and $\mu$ have almost the same KMS bundles. More precisely, for each $\beta \in \mathbb R$ the simplices $S^\alpha_\beta$ and $S^\mu_\beta$ of $\beta$-KMS states for $\alpha$ and $\mu$, respectively, are strongly affinely isomorphic in the sense of \cite{Th2} when $\alpha$ and $\mu$ are cocycle-conjugate. It follows that if $\alpha$ and $\mu$ are weakly cocycle-conjugate, there is a $\lambda \in \mathbb R \backslash \{0\}$ such that $S^\alpha_\beta$ is strongly affinely homeomorphic to $S^\mu_{\lambda \beta}$ for all $\beta \in \mathbb R$. Recall that by Theorem 3.2 of \cite{PS} an approximately inner flow on $A$ has KMS spectrum equal to $\mathbb R$, so to get uncountably many flows that are not mutually weakly cocycle-conjugate and also not approximately inner, there are many ways to go; for example the following. Let $X$ be a compact metric space. By Theorem \ref{21-10-21a} there is a flow $\theta^X$ on $A$ such that $S^{\theta^X}_\beta = \emptyset$ when $\beta \notin \{0,1\}$ while $S^{\theta^X}_1$ is affinely homeomorphic to the simplex $M(X)$ of Borel probability measures on $X$. Since $M(X)$ is not strongly affinely isomorphic to $M(Y)$ when $X$ is not homeomorphic to $Y$, the flow $\theta^X$ is not weakly cocycle-conjugate to $\theta^Y$ when $X$ is not homeomorphic to $Y$. Hence the cardinality of weak cocycle-conjugacy classes of flows on $A$ that are not approximately inner exceeds the cardinality of homeomorphism classes of compact metric spaces.\end{proof}

It is known that many simple monotracial $C^*$-algebras absorb the Jiang-Su algebra tensorially and so are covered by Corollary \ref{21-10-21b}. In particular, this is the case when $A$ in addition is nuclear and has the strict comparison property; cf. \cite{MS1} and \cite{MS2}.

\section{KMS bundles for flows on a unital infinite $C^*$-algebra}

In this section we prove the following statement which is a complement to the main result of \cite{ET}.

\begin{thm}\label{16-08-21cx} Let $A$ be a separable, simple, unital, purely infinite and nuclear $C^*$-algebra in the UCT class. Assume that $K_1(A)$ is torsion free. Let $(S,\pi)$ be a proper simplex bundle such that $0 \notin \pi(S)$. There is a $2 \pi$-periodic flow $\theta$ on $A$ whose bundle of KMS states is isomorphic to $(S,\pi)$.
\end{thm} 

Since a flow on a unital $C^*$-algebra without trace states cannot have $0$-KMS states and since in any case the KMS states form a proper simplex bundle (see Section \ref{2000} above and \cite{ET}), we have the following corollary.

\begin{cor}\label{16-08-21c} Let $A$ be as in Theorem \ref{16-08-21cx} and let $D$ be a separable unital $C^*$-algebra without trace states and $\theta'$ a flow on $D$.  There is a $2 \pi$-periodic flow $\theta$ on $A$ such that $(S^{\theta'},\pi^{\theta'})$ is isomorphic to $(S^\theta, \pi^\theta)$. 
\end{cor}

\subsection{The proof of Theorem \ref{16-08-21cx}}


Let $(S,\pi)$ be a proper simplex bundle such that $0 \notin \pi(X)$. If $S = \emptyset$ we can take the $\theta$ to be the trivial flow; so we assume that $S \neq \emptyset$. Choose a countable subgroup $\mathcal C$ of $\mathcal A(S,\pi)$ with the following properties.

\begin{itemize}
\item[(1)] The support of each element of $\mathcal C$ is compact.
\item[(2)] For every $N \in \mathbb N$, every $\epsilon > 0$ and every $f \in \mathcal A(S,\pi)$ such that $\supp f \in \pi^{-1}([-N,N])$, there is an element $g \in \mathcal C$ such that $\supp g \subseteq \pi^{-1}([-N,N])$ and
$$
\sup_{x \in S} \left|f(x)-g(x)\right| < \epsilon .
$$
\item[(3)] The function
$$
x \mapsto  e^{n \pi(x)}{(1-e^{-\pi(x)})^m}f(x)
$$
is in $\mathcal C$ for all $f \in \mathcal C$ and all $n,m\in \mathbb Z$.
\end{itemize}

Let $1_+$ and $1_-$ denote the characteristic functions of $[0,\infty)$ and $(-\infty,0]$, respectively. Since $0 \notin \pi(X)$ the functions $1_\pm \circ \pi$ are both continuous on $S$. Let $G_0 \subseteq \mathcal A(S,\pi)$ denote the subgroup generated by $\mathcal C$ and the functions
$$
x \mapsto (1_\pm \circ \pi)  e^{n \pi(x)}{(1-e^{-\pi(x)})^m}, \ n,m \in \mathbb Z.
$$
Set $G = \mathbb QG_0$ and
$$
G^+ = \left\{ f \in G: \ f(x) > 0, \  x \in S\right\} \cup \{0\} .
$$
\begin{lemma}\label{01-09-21} The pair $(G,G^+)$ has the following properties.
\begin{itemize}
\item[(1)] $G^+ \cap (-G^+) = \{0\}$.
\item[(2)] $G = G^+ - G^+$.
\item[(3)] $(G,G^+)$ is unperforated, i.e., $n \in \mathbb N \backslash \{0\}, \ g \in G, \ ng \in G^+ \Rightarrow g \in G^+$. \item[(4)] $(G,G^+)$ has the strong Riesz interpolation property, i.e. if 
$f_i, g_j$, $ i,j \in \{1,2\}$, are elements of $G$ and $f_i < g_j$ in $G$ for all $i,j \in \{1,2\}$, then there is an element $h \in G$ such that
$$
f_i < h < g_j
$$
for all $i,j \in \{1,2\}$. 
\end{itemize}
\end{lemma}
\begin{proof} The first three items are easy to establish. To prove (4) let $f_1,f_2,g_1,g_2 \in G$ such that $f_i < g_j$ in $G$ for all $i,j \in \{1,2\}$. By definition of $G$ there is $N \in \mathbb N$ so large that there are polynomials $p_1,p_2,q_1,q_2$ with rational coefficients such that
$$
\left(e^{\pi(x)}(e^{\pi(x)} -1)\right)^N f_i(x)  =  p_i(e^{\pi(x)})
$$
and
$$
\left(e^{\pi(x)}(e^{\pi(x)} -1)\right)^N g_j(x)  =  q_j(e^{\pi(x)})
$$
for all $i,j$ and all large $x$. It follows then from Lemma 4.7 of \cite{ET} that there is a polynomial $p_+$ with rational coefficients such that the function
$$
h_+(x) = 1_+ \circ \pi(x) \left(e^{\pi(x)}(e^{\pi(x)} -1)\right)^{-N} p_+(e^{\pi(x)}) 
$$
is an element of $G$ with the property that there is $K_+ > 0$ such that $h_+(x) = 0$ when $\pi(x) \leq 0$ and
$$
f_i(x) < h_+(x) < g_j(x) 
$$
for all $i,j$ and all $x \in S$ with $\pi(x) \geq K_+$. In the same way we get also an element $h_- \in G$ and a $K_- > 0$ such that $h_-(x) =0$ when $\pi(x) \geq 0$ and
$$
f_i(x) < h_-(x) < g_j(x) 
$$
for all $i,j$ and all $x \in S$ with $\pi(x) \leq - K_-$. Set $K = \max \{K_+,K_-\}$. It follows from (2) in Lemma 4.4 of \cite{ET} that there is $H\in \mathcal A(S,\pi)$ such that $H(x) = h^-(x)$ when $\pi(x) \leq -K$, $ H(x) = h^+(x)$ when $\pi(x) \geq K$ and
$
f_i(x) < H(x) < g_j(x)$
for all $x \in S$ and all $i,j$. Let $\delta > 0$ be smaller than $H(x) -f_i(x)$ and $g_j(x) - H(x)$ for all $i,j$ and all $x \in \pi^{-1}([-K,K])$. Since $H-h_+-h_-$ is supported in $[-K,K]$ it follows from the definition of $\mathcal C$ that there is $g \in \mathcal C$ such that $\supp g \subseteq \pi^{-1}([-K,K])$ and $\left|g(x) - (H(x) - h_+(x) - h_-(x))\right| < \delta$ for all $s \in S$. Then 
$$
h = g + h_+ + h_- 
$$
is an element of $G$ with the desired property.

\end{proof}

In short, Lemma \ref{01-09-21} says that $(G,G^+)$ is a dimension group with the strong Riesz interpolation property. We define an automorphism $\alpha$ of $(G,G^+)$ such that
$$
\alpha( g) = e^{-\pi}g.
$$

Let $A$ be the algebra from Theorem \ref{16-08-21cx}.  By Proposition 3.5 of \cite{Ro1} there is a countable abelian torsion free group $H$ and an automorphism $\kappa$ of $H$ such that $K_1(A)\simeq \ker (\id- \kappa)$ and $\coker (\id - \kappa) \simeq K_0(A)$. We note that we may assume that $H \neq \{0\}$; if not we exchange $H$ with $\mathbb Q$ and set $\kappa(x) = 2 x$. Set
$$
G_\sharp = H \oplus {G} \ .
$$
Let $p : G_\sharp \to G$ be the canonical projection and set
$$
G_\sharp^+ = \left\{ x \in G_\sharp : \ p(x) \in G^+\backslash \{0\}\right\} \cup \{0\}.
$$
It follows from Lemma 3.2 of \cite{EHS} and Lemma \ref{01-09-21} above that $(G_\sharp,G^+_\sharp)$ is dimension group.

Define $\alpha_\sharp \in \Aut G_\sharp$ by
$$
\alpha_\sharp = \kappa \oplus \alpha .
$$
It follows from \cite{EHS} and \cite{E1} that there is a stable AF algebra $B$ such that $(K_0(B),K_0(B)^+) = (G_\sharp, G^+_\sharp)$ and an automorphism $\gamma \in \Aut B$ such that $\gamma_* = \alpha_\sharp$.

\begin{lemma}\label{01-09-21d} $(G_\sharp,G_\sharp^+)$ has large denominators; that is, for all $g \in G_\sharp^+$ and $m \in \mathbb N$ there is an element $h \in G_\sharp^+$ and an $n \in \mathbb N$ such that $m h \leq g \leq n h$ in $G_\sharp$. 
\end{lemma}
\begin{proof} It suffices to show that $(G,G^+)$ has large denominators and this is automatic because $\mathbb Q G = G$.
\end{proof}

It follows now from Lemma 3.4 of \cite{Th3} that we may arrange for $\gamma$ to have the following additional properties.
\begin{itemize}
\item[(A)] The restriction map 
$\mu \ \mapsto \ \mu|_{B}$
is a bijection from traces $\mu $ on $B \rtimes_{\gamma} \mathbb Z$ onto the $\gamma$-invariant traces on $B$, and
\item[(B)] $B \rtimes_{\gamma} \mathbb Z$ is $\mathcal Z$-stable; that is $(B\rtimes_{\gamma} \mathbb Z)\otimes \mathcal Z \simeq B \rtimes_{\gamma} \mathbb Z$ where $\mathcal Z$ denotes the Jiang-Su algebra, \cite{JS}.
\end{itemize}

\begin{lemma}\label{04-08-21x} The only order ideals $I$ in $G_\sharp$ such that $\alpha_\sharp(I) = I$ are $I = \{0\}$ and $I = G_\sharp$.
\end{lemma}
\begin{proof} Recall that an order ideal $I$ in $G_\sharp$ is a subgroup such that
 \begin{itemize}
\item[(a)] $I= I \cap G_\sharp^+ -  I \cap G_\sharp^+$, and
\item[(b)] when $ 0 \leq y \leq x$ in $G_\sharp$ and $x \in I$, then $y \in I$.
\end{itemize}
Let $I$ be a non-zero order ideal such that $\alpha_\sharp(I) = I$. Then $p(I)$ is an order ideal in $G$ such that $\alpha(p(I)) = p(I)$. Since $p(I) \cap G^+ \neq \{0\}$ there is an element $g \in p(I) \cap G^+ \backslash \{0\}$. By definition of $G$ there are natural numbers $n,m,K \in \mathbb N$ such that the function
$$
g'(x) = 1_-\circ \pi e^{n \pi} \ + \ 1_+ \circ \pi e^{-m \pi}
$$
has the property that
$$
0 < g'(x) < Kg(x), \ \ x \in S .
$$  
It follows that $g' \in p(I)$. Since $p(I)$ is $\alpha$-invariant it follows that 
$$
\alpha^l(g') = 1_-\circ \pi e^{(n-l) \pi} \ + \ 1_+ \circ \pi e^{-(m+l) \pi} \in p(I)
$$
for all $l \in \mathbb Z$. For an arbitrary element $g \in G^+$ we can find $l_1,l_2 \in \mathbb Z$ and $M \in \mathbb N$ such that 
$$
g(x) < M(\alpha^{l_1}(g')(x) \ + \ \alpha^{l_2}(g')(x)) 
$$
for all $x \in S$. This implies that $G^+ \subseteq p(I)$ and hence that $G = p(I)$. To conclude that $I = G_\sharp$ it only remains to show that $H \oplus 0 \subseteq G_\sharp$. For this note that there is an element $h'\in H$ such that $(h',1) \in I$. Note that, for any $h \in H$, $(h,0) = (h+h',1) - (h',1)$ and $0 < (h+h',1) < 2(h',1)$ in $G_\sharp$. It follows that $(h+h',1) \in I$ and hence that $(h,0) \in I$.
\end{proof}

It follows from Lemma \ref{04-08-21x} and \cite{E1} that there are no non-trivial $\gamma$-invariant ideals in $B$. Since $\gamma_*^k = \alpha^k$ is non-trivial for all $k \neq 0$, no non-trivial power of $\gamma$ is inner and it follows therefore from \cite{E2} that 
$$
C = B \rtimes_\gamma \mathbb Z
$$
is simple. (See also \cite{Ki1}.)

Let $q : H \to H/(\id - \kappa)(H) = K_0(A)$ be the quotient map and choose $w \in H$ such that $q(w) = [1]$. Let $v = (w,1) \in G_\sharp^+$.

\begin{lemma}\label{07-10-21g} Let $\phi : G_\sharp \to \mathbb R$ be a positive homomorphism and $s > 0$ a positive numbers such that $\phi(v) = 1$ and $\phi \circ \alpha_\sharp = s\phi$. Let $\beta = - \log s$. There is an element $\omega \in \pi^{-1}(\beta)$ such that $\phi(h,g) = g(\omega)$ for all $(h,g) \in G_\sharp$.  
\end{lemma}
\begin{proof} For $h \in H$, we find that $\pm n(h,0) + v \geq 0$ in $G_\sharp$ and hence $\pm n\phi(h,0) + 1 \geq 0$ for all $n \in \mathbb N$, implying that $\phi(h,0) = 0$. It follows that $\phi$ factorises through $p$, i.e., there is a positive homomorphism $\psi : G \to \mathbb R$ such that $\psi \circ p = \phi$. If $f \in G$, $n,m \in \mathbb N$ and $|f(x)| < \frac{n}{m}$ for all $x \in S$, it follows that $-n < mf < n$ in $G$. Since $\psi(1) = 1$ this implies that $\left|\psi(f)\right| \leq \frac{n}{m}$. It follows that $\psi$ is a $\mathbb Q$-linear contraction. Let $\mathcal A_\mathbb R(S,\pi)$ be the subspace of $\mathcal A(S,\pi)$ consisting of the elements of $\mathcal A(S,\pi)$ that have a limit at infinity and denote by $\mathcal A_0(S,\pi)$ the subset of $\mathcal A_\mathbb R(S,\pi)$ consisting of the elements of $\mathcal A_\mathbb R(S,\pi)$ that vanish at infinity. Every element of $\mathcal A_\mathbb R(S,\pi)$ can approximated in norm by elements of the form
$$
q1_-\circ \pi + q1_+\circ \pi + f
$$
where $q \in \mathbb Q$ and $f \in \mathcal A_0(S,\pi)$. It follows therefore from the second condition on $\mathcal C$ that every element of $\mathcal A_\mathbb R(S,\pi)$ can be approximated in norm by elements of $G$, implying that $\psi$ extends by continuity to a linear contraction $\psi : \mathcal A_\mathbb R(S,\pi) \to \mathbb R$. By the Hahn-Banach theorem there is a contractive extension of $\psi$ to all continuous real-valued functions on $S$ with a limit at infinity. Since $\psi(1) = 1$ this extension is positive and it follows therefore that there is a bounded Borel measure $m$ on $S$ such that
$$
\psi(f) = \int_S f(x) \ \mathrm{d}m(x) 
$$
for all $f \in \mathcal A_0(S,\pi)$. From here the argument from the last part of the proof of Lemma 4.12 of \cite{ET} can be repeated to obtain an element $\omega \in \pi^{-1}(\beta)$ such that $\psi(f) = f(\omega)$ for all $f \in G$, implying that $\phi(h,g) = g(\omega)$ for all $(h,g) \in G_\sharp$.
\end{proof}

Let $e\in B$ be a projection such that $[e] = v$ in $K_0(B) = G_\sharp$.

\begin{lemma}\label{07-10-21h} $eCe$ is $*$-isomorphic to $A$.
\end{lemma}
\begin{proof} It follows from the Pimsner-Voiculescu exact sequence, \cite{PV}, that $K_0(C)$ can be identified with the cokernel of $\id - \alpha_\sharp$ in such a way that the map $\iota_* : K_0(B) \to K_0(C)$ induced by the inclusion $\iota : B \to C$ becomes the quotient map $q : G_\sharp \to G_\sharp/(\id - \alpha_\sharp)(G_\sharp)$. It follows from the definition of $G$ that $\id - \alpha$ is an automorphism of $G$ and hence
$$
 G_\sharp/(\id - \alpha_\sharp)(G_\sharp) = H/(\id - \kappa)(H) = K_0(A) .
 $$
The resulting isomorphism $K_0(eCe) \to K_0(A)$ takes $[e]$ to $[1]$. Similarly, the Pimsner-Voiculescu exact sequence shows that 
$$
K_1(C) = \ker(\id - \alpha_\sharp)= \ker (\id - \kappa) = K_1(A) .
$$ 
We conclude that $K_1(eCe) \simeq K_1(A)$. The proof is then completed by using the Kirchberg-Phillips result; see (iv) of Theorem 8.4.1 of \cite{Ro3}. For this we need to verify that both algebras are separable, simple, unital, purely infinite, nuclear and in the UCT class. This is assumed for $A$ and regarding $eCe$ most of these properties are well-known. (For the UCT see 23.1.1 and 22.3.5 (g) of \cite{Bl}.) To see that $eCe$ is purely infinite we note that $eCe$ is $\mathcal Z$-stable thanks to \cite{TW} and (B) above. Furthermore, it follows from (A) above, in combination with Lemma 3.5 of \cite{Th3} and Lemma \ref{07-10-21g} above, that there are no traces on $eCe$. By Corollary 5.1 of \cite{Ro2} this implies that $eCe$ is purely infinite.
\end{proof}

We consider the dual action on $C = B  \rtimes_\gamma \mathbb Z$ as a $2\pi$-periodic flow and we denote by $\theta$ the restriction of this flow to $eCe$.

\begin{lemma}\label{02-09-21d} The KMS bundle $(S^\theta,\pi^\theta)$ of $\theta$ is isomorphic to $(S,\pi)$.
\end{lemma}
\begin{proof} Let $(\omega,\beta) \in S^\theta$. By Remark 3.3 in \cite{LN} there is a $\beta$-KMS weight $\hat{\omega}$ for the dual action on $C$ which extends $\omega$. By Corollary 4.2 of \cite{ET} the map
$$
\omega \mapsto \left(\hat{\omega}|_B\right)_*
$$
is an affine homeomorphism from the simplex of $\beta$-KMS states for $\theta$ onto the set of elements $\phi$ from $\Hom_+(G_\sharp,\mathbb R)$ such that $\phi \circ \alpha = e^{-\beta}\phi$ and $\phi(v) = 1$. By Lemma \ref{07-10-21g} there is a point $\omega'$ in $\pi^{-1}(\beta)$ such that $\left(\hat{\omega}|_B\right)_*(h,g) = g(\omega')$ for all $(h,g) \in G_\sharp$. Define $\varphi : S^\theta \to S$ by $\varphi(\omega,\beta) = \omega'$. Note that $\pi \circ \varphi = \pi^\theta$. To see that $\varphi$ is surjective, let $\mu \in S$. Define $\ev_\mu : G_\sharp \to \mathbb R$ by $\ev_\mu(h,g) = g(\mu)$. Then $\ev_\mu \in \Hom_+(G_\sharp,\mathbb R)$, $\ev_\mu \circ \alpha = e^{-\pi(\mu)}\ev_\mu$, and $\ev_\mu(v) = 1$. Using Corollary 4.2 of \cite{ET} we get a $\pi(\mu)$-KMS state $\omega$ for $\theta$ such that $\left(\hat{\omega}|_B\right)_* = \ev_\mu$. Then $\varphi(\omega,\pi(\mu)) = \mu$. To see that $\varphi$ is injective, consider $(\omega_i,\beta_i) \in S^{\theta}$ such that $\varphi(\omega_1,\beta_1) = \varphi(\omega_2,\beta_2)$. Then $\beta_1 = \pi((\varphi(\omega_1,\beta_1)) =\varphi(( \varphi(\omega_2,\beta_2)) = \beta_2$ and 
$$
\left(\widehat{\omega_1}|_B\right)_* = \left(\widehat{\omega_2}|_B\right)_* .
$$ 
Since $B$ is AF it follows that $\widehat{\omega_1}|_B = \widehat{\omega_2}|_B$; see Lemma 3.5 of \cite{Th3}. By Lemma 3.1 of \cite{Th3} it follows that $\widehat{\omega_1} = \widehat{\omega_2}$ and hence $\omega_1 =\widehat{\omega_1}|_{eCe} = \widehat{\omega_2}|_{eCe} = \omega_2$. 

It remains to show that $\varphi$ is a homeomorphism. But since $\varphi$ is a bijection and $\pi \circ \varphi = \pi^\theta$ it suffices to show that $\varphi^{-1}$ is continuous. Let therefore $\{\mu^n\}$ be a sequence in $S$ such that $\lim_{n \to \infty} \mu^n = \mu$ in $S$. Set $\beta_n = \pi(\mu^n)$ and note that $\lim_{n \to \infty} \beta_n = \beta$, where $\beta = \pi(\mu)$. It follows that $\lim_{n \to \infty} \ev_{\mu^n}(x) = \ev_\mu(x)$ for all $x \in G_\sharp$. Let $\tau^n$ and $\tau$ be the traces on $B$ determined by the conditions that ${\tau^n}_* = \ev_{\mu^n}$ and $\tau_* = \ev_\mu$. Then $\varphi^{-1}(\mu^n) = (\tau^n\circ P|_{eCe},\beta_n)$ and $\varphi^{-1}(\omega) = (\tau \circ P|_{eCe},\beta)$. It suffices therefore to show that 
\begin{equation}\label{07-12-21}
\lim_{n \to \infty} \tau^n\circ P(exe) = \tau \circ P(exe)
\end{equation} 
for all $x \in C$. Since $\tau^n\circ P(e) = \tau \circ P(e) =1$, the positive functionals on $C$ given by $x \mapsto \tau^n\circ P(exe)$ and $x \mapsto \tau\circ P(exe)$ are all of norm $\leq 1$. To establish \eqref{07-12-21} for all $x \in C$ it suffices therefore to check the equality for $x$ in a dense subset of $C$. If $u$ is the canonical unitary in the multiplier algebra of $C$ coming from the construction of $C$ as a crossed product, it suffices to show that $\lim_{n \to \infty} \tau^n\circ P(ebu^ke) = \tau \circ P(ebu^ke)$ for all $k \in \mathbb Z$ and all $b \in B$. Since $P(ebu^ke) = 0$ when $k \neq 0$ it suffices to consider the case $k = 0$; that is, it suffices to show that
$\lim_{n \to \infty} \tau^n (ebe) = \tau(ebe)$. On approximating $ebe$ by a sum of projections from $eBe$ it suffices to show that $\lim_{n \to \infty} \tau^n(p) = \tau(p)$
when $p$ is a projection in $eBe$. This holds because 
$$
\lim_{n \to \infty}\tau^n(p)  = \lim_{n \to \infty} \ev_{\mu^n}([p]) = \ev_\mu([p]) = \tau(p).
$$

\end{proof}

Theorem \ref{16-08-21cx} follows from Lemma \ref{02-09-21d} and Lemma \ref{07-10-21h}.

\end{document}